%% file: main.tex
\documentclass[11pt,reqno, oneside]{amsart}

\input{commands}

\usepackage[utf8]{inputenc}
\usepackage[english]{babel}
\usepackage{amsmath}
\usepackage{amssymb}
\usepackage{tikz-cd}
\usepackage{mathabx}
\usepackage{amsfonts}
\usepackage{amsthm}
\usepackage[shortlabels]{enumitem}
\usepackage[dvipsnames]{xcolor}
\usepackage{tikz-cd}
\usepackage{subfiles}
\usepackage[mathscr]{euscript}
\usepackage{extarrows}
\usepackage[pagebackref=true] {hyperref}
\usepackage{appendix}
\usepackage{bm}
\usepackage{mathtools}
\usepackage{float}
\usepackage{pdfpages}
\usepackage{transparent}
\usepackage{changepage}
\usepackage{wrapfig}
\usepackage{todonotes}
\usepackage{relsize}
\usepackage{dsfont}
\hypersetup{
    colorlinks=true,
    linkcolor=Blue,
    filecolor=red,      
    urlcolor=Blue,
    citecolor=Blue,
}

\newtheorem{theorem}{Theorem}[section]

\theoremstyle{definition}
\newtheorem{definition}[theorem]{Definition}

\theoremstyle{remark}
\newtheorem{remark}[theorem]{Remark}

\newtheorem*{notation*}{Notation}

\makeatletter
\@addtoreset{equation}{section}
\@addtoreset{theorem}{section}
\makeatother

\makeatletter %%\def\l@subsection{\@tocline{2}{0pt}{1pc}{5pc}{}}
 \def\l@subsection{\@tocline{2}{0pt}{2pc}{6pc}{}} \makeatother

  \bigskip
  \footnotesize

\begin{document}

\title{Classical cylindrical contact homology is not invariant over $\bZ$}

\author{Soham Chanda}
\date{\today}

\maketitle
\begin{abstract}
   We show that classical cylindrical contact homology over $\mathbb{Z}$, equipped with its decomposition by free homotopy classes, is not invariant under changes of hypertight contact form. Specifically, we construct two nondegenerate hypertight perturbations of the standard Seifert contact form on $\Sigma(2,4,5)$ whose cylindrical contact homologies in the regular-fiber class are respectively torsion-free and contain $\mathbb{Z}/5$-torsion.
\end{abstract}

\tableofcontents

\input{example}

\bibliographystyle{amsalpha}
\bibliography{bib}

\end{document}

%% file: commands.tex
\newcommand{\Mbar}{\overline{\mathcal M}}

\newcommand{\bC}{\mathbb{C}}

\newcommand{\bH}{\mathbb{H}}

\newcommand{\bQ}{\mathbb{Q}}
\newcommand{\bR}{\mathbb{R}}

\newcommand{\bZ}{\mathbb{Z}}
\newcommand{\cP}{\mathcal{P}}

\newcommand{\cS}{\mathcal{S}}

\newtheorem{intthm}{Theorem}

\DeclareMathOperator{\glue}{\normalfont\text{glue}}
\DeclareMathOperator{\forget}{\normalfont\text{forget}}

%% file: example.tex
\section{Introduction }

% Just like the various versions of orbifold Morse homology are non-equivalent \cite{BaoLiu2025,ChoHong2014}, the positive and negative version of the cylindrical contact homologies are non-equivalent. By this we mean that there is an example where the two homology theories differ. In particular, we will exhibit that $CCH^+$ will be torsion-free while $CCH^-$ will have torsion for the Brieskorn manifold $\Sigma(2,4,5)$.

The first published instance of contact homology can be traced to Eliashberg in \cite{eliashberg1998invariants} and reappears in the seminal paper \cite{egh00}. 
% Given a contact manifold $(Y,\lambda)$, we denote the set of \textit{good}\footnote{Being good is a property how cyclic reparameterization of Reeb orbits affect orientation lines.} Reeb orbits as $\cP_{good}$. 
When a contact manifold $(Y,\xi)$ is equipped with a contact form $\lambda$ and an admissible almost complex $J$ on the symplectization $\hat Y$, the contact DGA is defined as the free supercommutative dg algebra generated by $\cP_{good}$ where the differential is constructed by counting rigid punctured $J$-holomorphic genus zero curves with exactly one input. The homology of the contact dga over rational coefficients is an invariant up to isotopy of contact structure and almost complex structure \cite{pardon2019contact,bao2023semi}.

There is a simpler construction of an invariant homology theory for \textit{hypertight} contact form $\lambda_\circ$, i.e. when there are no contractible Reeb orbits which was outlined in \cite[\S 1.9.2]{egh00}. The chain complex in this case is simply generated by elements of the set of good Reeb orbits, $\cP$, $$CC(Y,\lambda_\circ,J) = \bigoplus_{\gamma \in \cP}\bQ \langle\gamma \rangle,$$and the differential is obtained from counting rigid $J$-holomorphic cylinders. The non-contractibility of Reeb orbits ensure that the SFT compactification of moduli of stable cylinders consists of cylindrical SFT buildings. The details of cylindrical contact homology got developed later in \cite{yau2004cylindrical} for the subcritical case and in \cite{baohonda1,hutchings2014cylindrical} for the three dimensional case using automatic transversality techniques inspired  from \cite{Wendl2010Automatic}. The introduction of virtual techniques of \cite{pardon2019contact,bao2023semi} enabled the construction of cylindrical contact homology for any hyper-tight contact manifold.

\subsection{Context}

For a non-hypertight contact form, due to codimension one breaking of SFT spheres, one needs to incorporate multiple output curves in the differential. Although one might try to order the outputs, in the boundary of a (virtual) one dimensional moduli space, the Reeb orbits that appear in between two levels do not come in order. The classical fix to this issue is to use the supercommutative dga generated by Reeb orbits which forces a $S_n$ symmetry, which in turn requires a $\frac{1}{n!}$ term in the definition of the differential. Hence in the non-hypertight case, the classical contact homology definition necessarily needs $\bQ$ coefficients instead of $\bZ$.

On the other hand, for cylindrical contact homology, \cite{hutchings2014cylindrical,baohonda1} notice that the differential counts for cylindrical contact chain complexes are integers. Moreover, the moduli spaces that are used in defining the contact differential can be  derived \textit{manifolds}, instead of being derived \textit{orbifolds} as discussed in Remark \ref{refl}. Choosing to work over rationals was \textit{sufficient} for proofs of invariance of cylindrical contact homology in the literature, but it was unknown if such a choice was \textit{necessary}.

The issue of defining contact homology with integer coefficients has remained unresolved in the community, as noted in \cite[Pg 15]{pardon2019contact}. One definitely needs to resolve an issue of how to bypass the need of the $\frac{1}{n!}$ for using symmetrized words in contact homology but such worries become immaterial for cylindrical contact homology. In his blog post \cite{Hutchings2014IntegerCoefficients}, Hutchings outlines a heuristic based on some conjectural curve counting results that a certain chain map in cylindrical contact homology in the three dimensional case needs to be defined using rational coefficients.  

\subsection{Main Results}
In this article we show that the \textit{classical} (as proposed in \cite{EliashbergGiventalHofer2000}) construction of cylindrical contact homology is not invariant under change of hypertight contact forms with $\bZ$ coefficients. We will use computations from \cite{haney2022cylindrical} to get a filtered homology computation.

\begin{intthm}[Theorem \ref{thm:comp}]\label{int:A}
     The Brieskorn manifold, $\Sigma(2,4,5)$, with its standard Morse-Bott contact form, has two families of hypertight contact forms $\lambda^T_{\pm}$ that are isotopic to each other, but the corresponding filtered ($<T$) cylindrical contact homologies are not quasi-isomorphic to each other since one has torsion while the other does not. 
\end{intthm}

\noindent 

Recall that cylindrical contact chain complexes naturally split by free homotopy classes. All the proofs of invariance in Floer theory are functorial under change of data. For the particular case of contact homology, the functoriality preserves the splitting by free homotopy classes of loops. By using genericity arguments, we show there are choices of data that produce different cylindrical contact homologies in certain free homotopy class. 

\begin{intthm}[Theorem \ref{thm:finale}]
  There are $C^2$ close hypertight contact forms near the Morse-Bott contact form $\lambda$ on $\Sigma(2,4,5)$ and admissible almost complex structures with enough regularity such that the cylindrical chain homologies, filtered by free homotopy classes, are not quasi-isomorphic over integers.
\end{intthm}

%  Thus, one might be tempted to believe that cylindrical contact homology can be \textit{easily} defined over integers. 

% % where $\cP$ is the  and the differential counts 

% One reason contact homology is defined over $\bQ$ instead of $\bZ$ stems from the fact that there are multiple outputs and we use symmetric words to read the output. 

% Interestingly, this is not the case when you require your definition to be invariant of your choice of contact form. In this note we share an example that exhibits the following can't be true :

% \begin{quote}
%     For a hypertight contact manifold $(Y,\lambda)$, there is a chain complex, $(CC(Y,\lambda),\partial)$ constituted by free-$\bZ$ module generated by good Reeb orbits, and differential from ``counting points of the derived Manifold of moduli of \textit{asymptotically marked} $J$-holomorphic cylinder buildings " whose homology over $\bZ$ is \textit{independent of } a hypertight $\lambda$.
% \end{quote}

\noindent In particular, this proves that the `classical' cylindrical contact homology package, equipped with functorial morphisms from exact cobordisms,  introduced in \cite{egh00} and later rigorously constructed in  \cite{baohonda1, hutchings2014cylindrical,hutchings2022s}
% 
% \footnote{Hutchings-Nelson's paper \cite{hutchings2022s}
%  has many versions of contact homology. I am talking about their $\partial_+ = \delta \kappa$ and $\partial_-=\kappa \delta$ version. They prove that the $S^1$ equivariant contact homology is actually an invariant over $\bZ$. } 
 % 
 is not invariant over $\bZ$  while preserving the splitting by free loop.

\subsection{Acknowledgements}
This article stemmed from an interaction with Michael Hutchings at ICM 2026, he introduced the two versions of cylindrical contact homology and the open question of whether those chain complexes are quasi-isomorphic. Thus, I gratefully acknowledge the support the National Science Foundation under Grant No. 2529637 which provided  travel support for participation at the ICM. I thank Michael Hutchings for the question and John Pardon for a valuable discussion. I would also like to thank Eric Kilgore for numerous conversations that helped me get a clearer picture of virtual curve counting using abstract perturbations and form my understanding as collected in Remark \ref{refl}.

\subsection{AI Disclosure}
ChatGPT 5.6 Sol was used to convert a hand-drawn sketch to an svg file for Figure \ref{fig:s} and to verify computations in \S \ref{subsec:comp}. It was also used for literature review, spellcheck, and grammar check. No part of the text except some LaTeX equations and diagrams was generated using LLMs.

 % \section{Brieskorn Manifolds and Cylindrical Contact Homology}

\section{Background on cylindrical contact homology}

We briefly recall the construction of cylindrical contact homology as treated in \cite{hutchings2014cylindrical,baohonda1}.
Given a triple $(Y,\lambda,J)$ of a hypertight contact manifold $(Y,\lambda)$ with a $\lambda$-admissible almost complex structure, we describe some relevant moduli spaces of $J$ holomorphic cylinders. Fix choices of base points $b_\gamma$ for all unparameterized Reeb orbits $[\gamma]$, equivalently fix parameterizations $\gamma :S^1 \to Y$ for all Reeb orbits.
Denote the SFT moduli space of $J$ cylindrical buildings between Reeb orbits $\gamma_+,\gamma_-$ with a positive asymptotic marker constrained to the base point $b_{\gamma_+}$ as 
$$\Mbar^+(\gamma_+,\gamma_-).$$

% $$\Mbar^+(\gamma_+,\gamma_-) = \{ u:\bR \times S^1 \to \bR \times Y | \delbar_Ju=0, \lim_{t\to \infty} u_Y(s,t)=\gamma_+(t),  , \lim_{t\to -\infty} u_Y(s,t)=\gamma_-(t+\theta) \text{ for some } \theta   \}/\sim$$
% where $u\sim u'$ if the maps are related by the domain $\bR$ reparameterization or a target $\bR$ translation. Note that this is equivalent to

\begin{definition}[Enough regularity]
    An admissible almost complex $J$ structure has \textit{enough regularity} when all the moduli spaces $\Mbar^+(\gamma_+,\gamma_-)$ of $J$-holomorphic cylinder-buildings of virtual dimension 1 and lower are regularly cut out.
\end{definition}
\begin{theorem}[\cite{hutchings2014cylindrical,baohonda1}]
    A generic choice of $J$ has enough regularity for three dimensional contact manifolds.
\end{theorem}
\noindent The cylindrical contact chain complex can be defined for a generic choice of $J$ due to automatic transversality results in dimension three \cite{hutchings2014cylindrical,baohonda1}.
% \footnote{i.e. a module with a prescribed morphism $\partial$}
For such a judicious choice of $J$, the contact chain complex over the ring $R$ is $$CC(Y,\lambda,J) = \bigoplus_{\gamma\in \cP_{good}} R\langle \gamma\rangle, \quad \partial(\gamma^+) = \sum \#\Mbar^+(\gamma_+,\gamma_-)\gamma_-$$
where $\# \Mbar$ is the algebraic count of the points of the moduli space, which by definition is 0 for components with nonzero virtual dimension. Invariance has been proved for the cylindrical case in \cite{baohonda1} and will appear in \cite{HutchingsNelsonInvariance} using traditional transversality techniques. We note that Pardon's virtual techniques\cite{pardon2019contact} also supply the necessary invariance statements over rational coefficients.
\begin{remark}
    There is an analogous  version of the chain complex with counts of \textit{negatively} marked cylinders as observed in \cite{nelson2015automatic,hutchings2022s}. In their language, the define $\partial=\delta \kappa$ where $\kappa(\gamma)$ is the multiplicity of the Reeb orbit and $\delta$ is the count of rigid cylinders \textit{without asymptotic marker.}
    Note that our definition of $\partial$ is just a geometric way of defining the same differential. The negative marked version would correspond to the differential $\kappa\delta$. The main construction of this article can also be used to show that these two versions of cylindrical contact homologies are not isomorphic over $\bZ$.
\end{remark}

\section{Brieskorn manifolds and computations}
 
The choice of using Brieskorn manifolds as the main example comes from trying to utilize that orbifold Morse homologies \cite{ChoHong2014,BaoLiu2025} are not invariant over integers without restricting to special subclasses of Morse functions.\footnote{Bao-Liu \cite{BaoLiu2025} use a special subclass of Morse functions that they call stable Morse functions and have conjectured invariance of certain homology theories.} Thus the way to find a contact manifold that satisfies the claims of Theorem \ref{thm:comp} one strategy is the following

\begin{itemize}
    \item Pick a smooth prequantum bundle $(Y,\lambda)$ over an orbifold, i.e. $Y$ is a smooth contact manifold with $S^1$ action inducing the Reeb flow such that $Y/S^1$ is an orbifold. Clearly $\lambda$ is a Morse-Bott contact form.
    
    \item Pick a Morse function $H$ with values in $[0,1]$ on $B=Y/S^1$ such that the $\bZ$ Morse homology of $H$ and $1-H$ are different.
    \item Show that the rigid cylinders in $\bR \times Y$ for the $\epsilon H-$perturbation and $\epsilon (1-H)$-perturbation of $\lambda$ are in correspondence with Morse flows on the base and thus arrive at a contradiction.
\end{itemize}
We show below that $\Sigma(2,4,5)$ can be used to implement our strategy.
\subsection{Cylindrical contact homology of Brieskorn manifolds}
We recall the results from \cite[\S 5]{haney2022cylindrical} about counting rigid pseudoholomorphic cylinders in symplectizations of  of Brieskorn manifolds.

\begin{definition}[Brieskorn 3-manifolds]
    The Brieskorn manifold for the triple of integers $(p,q,r)$ is the manifold 
    $$\Sigma(p,q,r) = \{ (x,y,z) \in \bC^3 | \; x^p + y^q + z^r = 0 \} \cap \{ |x|^2 + |y|^2 + |z|^2 = 1 \}.$$
\end{definition}

\noindent When $\frac{1}{p } + \frac{1}{q} + \frac{1}{r} < 1$, $\Sigma(p,q,r)$ the Milnor fillable contact structure admits a hypertight contact form $\lambda$. The Brieskorn manifolds for $(p,q,r)$ can be viewed as a principal $S^1$ bundle over an orbifold surface $S(p,q,r)$ via the Seifert fibration (\cite{neumann2006seifert,neumann2007graph}) with $s_j$ orbifold points of order $t_j$ where $j\in \{ 1,2,3 \}$ given by the relations
\[
\begin{aligned}
s_1 &= \frac{qr}{\operatorname{lcm}(q,r)},
& t_1 &= \frac{\operatorname{lcm}(p,q,r)}
               {\operatorname{lcm}(q,r)},\\
s_2 &= \frac{pr}{\operatorname{lcm}(p,r)},
& t_2 &= \frac{\operatorname{lcm}(p,q,r)}
               {\operatorname{lcm}(p,r)},\\
s_3 &= \frac{pq}{\operatorname{lcm}(p,q)},
& t_3 &= \frac{\operatorname{lcm}(p,q,r)}
               {\operatorname{lcm}(p,q)}.
\end{aligned}
\]
The standard contact form $\lambda$ on the Brieskorn manifolds are Morse-Bott, so we will perturb them using Morse functions. We collect the main result from \cite[\S 5]{haney2022cylindrical} that they use to compute cylindrical contact homology over $\bQ$ .
\begin{theorem}[Theorem 4.5, Corollary 5.1 \cite{haney2022cylindrical}] \label{thm:seb_coll}
    There is a function $H$ on $S(p,q,r)$, $H:S(p,q,r)\to[0,1],$ that is Morse away from the orbifold points,  such that
    \begin{itemize}
        \item it has unique maxima $y$,
        \item it has minima exactly at the orbifold points $\{v_{ij}\}$, $i\in \{ 1,2,\dots,s_j\}$ \item near the orbifold points $H$ is modeled by $|\phi|^2$ for certain automorphic forms $\phi$ on $\bH \times \widetilde {\bC^\times}$.
    \end{itemize}
    Let $x_1,\dots , x_r$ denote the critical points of index one. For every $T>0$ there is an $\epsilon>0$ such that the following holds.  The contact form $(1+\epsilon\pi^*H)\lambda$ has non-degenerate orbits upto action $T$ with simple Reeb orbits  $\gamma_{\ast}$, over the critical point $\ast$ of the function $H$. There is a choice of cylindrical almost complex structure $J_\epsilon$ that has enough regularity for computing cylindrical contact homology and the only rigid contributions are from lifts of gradient trajectories of $H$.
\end{theorem}

\begin{remark}\label{rem:anymorse}
    The existence of the admissible almost complex structure $J_\epsilon$ with rigid cylinder count given by gradient trajectory of $H$ does not depend on the special choice of the Morse function. The results in \cite[\S 5.1]{nelson2020automatic} prove that for any Morse function on the base, there is a choice of $J_\epsilon$ that makes rigid cylinder counts (up to a certain action) same as rigid gradient trajectories on the base.
\end{remark}
% \subsection{Positive and negative cylindrical homologies of $\Sigma(2,4,5)$}

 \subsection{Computations for $\Sigma(2,4,5)$}\label{subsec:comp}
 The base surface of $\Sigma(2,4,5)$ is a genus zero orbifold, whose isotropy can be computed from the relations above

\[
\begin{aligned}
s_1 &= 1 &&\text{usual manifold point,}\\
s_2 &= 1 &&\text{orbifold point of order } t_2=2,\\
s_3 &= 2 &&\text{orbifold points of order } t_3=5.
\end{aligned}
\]

\noindent The choice of Morse function as in Theorem \ref{thm:seb_coll} for the special case of $(2,4,5)$ has $4$ minima, with isotropy groups $\bZ/2\bZ$ for one of them and $\bZ/5\bZ$ for two of them and the other minima is a smooth point. The genus is zero and the Morse function has unique maxima, thus there are 3 index one critical points. Since the only rigid cylinders in $\bR \times \Sigma(2,4,5)$ are Morse trajectories, we can  compute the differentials explicitly.

\begin{figure}
    \centering
   \includegraphics[scale=0.5]{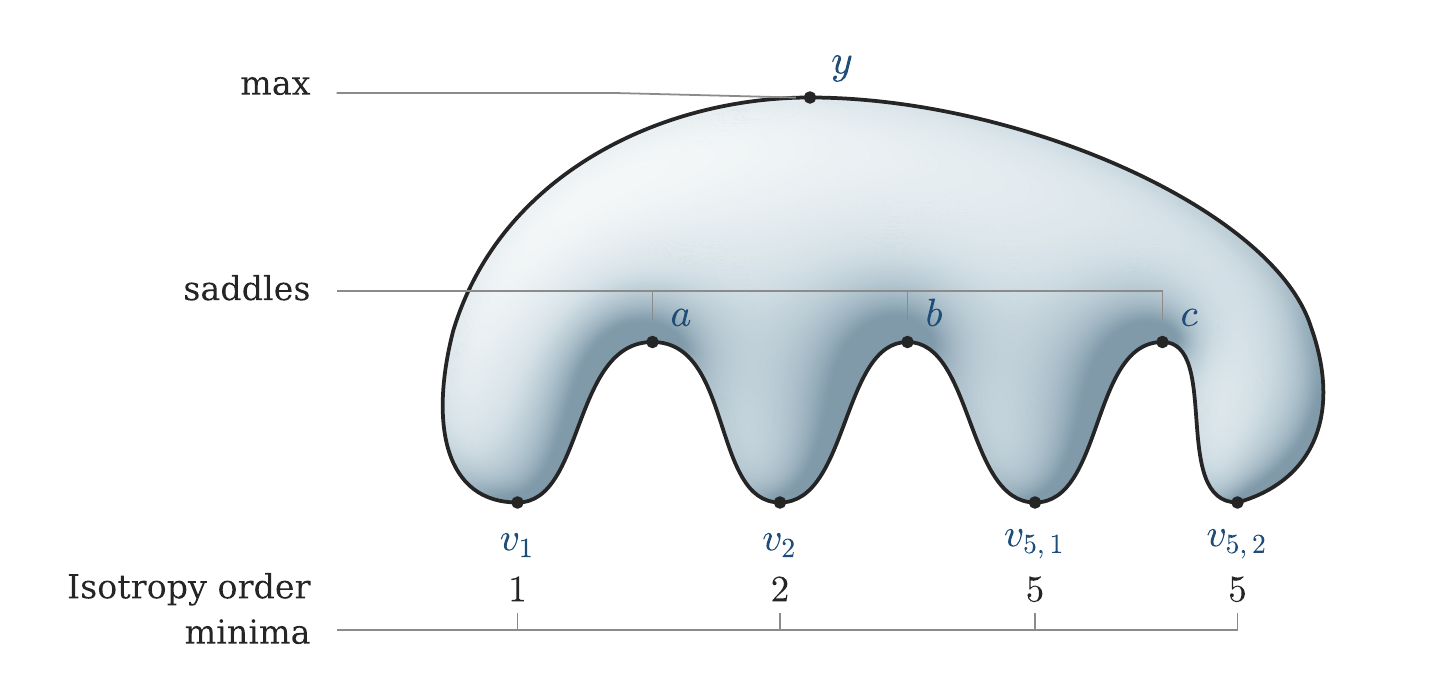}
    \caption{The function $H$  on $S(2,4,5)$. This figure was converted from handdrawn sketch to svg by ChatGPT 5.6. }
    \label{fig:s}
\end{figure}

%  Going from lower to higher action filtration, say from $T$ to $T+\delta$, does not introduce new rigid cylinders whose outputs are seen in the $T$ filtration. Recall the fact that homology and direct limits commute. To show that two chain complexes $C_+, C_-$ are not quasi-isomorphic, it is enough to prove that for large enough filtration level $T$, $C^{<T}_+$ has torsion in homology whereas $C^{<T}_-$ is torsion free.

Since we are in the Morse-Bott case, to recover the full contact homology one needs to take limit over increasing actions. We first exhibit the difference of homologies for any action larger that the action of a regular fiber (i.e. a fiber that is not over an orbifold point) of the circle fibration. Indeed for any fixed $T$, there is a $\epsilon_T>0$ such that all action-less-than-$T$-Reeb orbits of $(1+\epsilon_T\pi^*H)\lambda$ are non-degenerate.

\begin{theorem}\label{thm:comp}
With the choice of almost complex as outlined in Theorem \ref{thm:seb_coll} and Remark \ref{rem:anymorse}, the filtered cylindrical contact homology $CCH^{<T}(\Sigma(2,4,5),(1+\epsilon_T\pi^*H)\lambda)$ has no torsion and $CCH^{<T}(\Sigma(2,4,5),(1+\epsilon_T\pi^*(1-H))\lambda)$ has $\bZ/5\bZ$ torsion.
\end{theorem}

\begin{proof}

We utilize the fact that the only differentials are the lifts of Morse gradients and perform a partial computation.  Set the following notations for Morse critical points 
    \begin{itemize}
        \item $y$ is the unique maxima,
        \item $v_1$, $v_2,v_{5,1},v_{5,2}$, are minimas with the convention that $v_\ast$ has isotropy $\bZ / \ast \bZ,$
        \item $a,b,c$ are saddles.
    \end{itemize}
    When we perturb by $\epsilon H$, we have the following matrices denoting linear maps $\bZ^3 \to \bZ^4$, from the saddles to the minima \[
\partial_{H}
=
\begin{pmatrix}
1 & 0 & 0 \\
-1 & 1 & 0 \\
0 & -1 & 1 \\
0 & 0 & -1
\end{pmatrix}
\] 
When we perturb by $\epsilon(1-H)$ instead, we have that the matrix from the maxima to the saddle points
\[
\partial_{1-H}
=
\begin{pmatrix}
1 & -2 & 0 & 0 \\
0 & 2 & -5 & 0 \\
0 & 0 & 5 & -5
\end{pmatrix}
\]
Note that  all other differentials vanish, thus these matrices are the only non-trivial chain differentials. Clearly one matrix creates torsion in cokernel while the other does not.
\end{proof}

Denote the set of free homotopy classes $S$ that support the torsion $\bZ/5\bZ$ in the minimal choice of $T$ as $\cS_{tor}$. This can be checked to be equal to the singleton set containing the regular fiber class of the Seifert fibration. On the other hand, given any cylindrical chain complex, $CC(Y,\lambda,J)$ and a subset $\cS$ of the free homotopy class of loops, the subcomplex supported on $\cS$ is simply the subcomplex of Reeb orbit generators which present one of the elements in $\cS$. 

\begin{theorem}\label{thm:finale}
There are $C^2$ close hypertight contact forms near the Morse-Bott contact form $\lambda$ on $\Sigma(2,4,5)$ and admissible almost complex structures with enough regularity such that the cylindrical chain complexes are not quasi-isomorphic over integers via chain maps that preserve the free homotopy class of generators. In particular, the homology of the subcomplexes supported on $\cS_{tor}$ are different.
\end{theorem}
\begin{proof}
The outline of the proof is to start from the particular choices of almost complex structure and contact forms required for the partial computation in Theorem \ref{thm:comp} and then approximate the choices by non-degenerate data that preserve the filtered homology computation.

One can choose $T$ larger than a very big multiple of the regular fiber action in Theorem \ref{thm:seb_coll} and the computation of Theorem \ref{thm:comp} shows that one chain complex has torsion in low degrees while the other doesn't. We will use this as a starting point, and perturb the contact form and almost complex structure to finish our proof.
    
    From genericity results, we can find non-degenerate contact forms close to $(1+\epsilon_T\pi^*H)\lambda$ and $(1+\epsilon_T\pi^*(1-H))\lambda$ that preserve the set of Reeb orbits of action at most $T$. We can check that for $\Sigma(2,4,5)$, small perturbations preserve hypertightness. Indeed from \cite{haney2022cylindrical} we have that $dt(R_{f\lambda}) =  \frac{1}{f} + y\partial_y (\frac{1}{f})$ in the universal cover $\bH \times \bR$. Thus by choosing $f$ close to $1$, and bounds from the hyperbolic metric $dt^2 + \frac{1}{y^2}(dx^2 + dy^2)$ one can find $c$ such that $dt(R_{f\lambda})>c$. Thus
     we get that no contractible orbits can appear in $f\lambda$ for $f$ close to $1$. 
    
    We denote the non-degenerate perturbations of the forms $(1+\epsilon_T\pi^*H)\lambda$ and $(1+\epsilon_T\pi^*(1-H))\lambda$ as $\lambda^{H/1-H}_\circ$. We can choose almost complex structures $J^{H}, J^{1-H}$ admissible with respect to $\lambda^{H/1-H}_\circ$  with enough regularity such that the cylinders appearing in the computation in Theorem \ref{thm:comp} persist. This can be ensured from usual persistence of regular curves.
    
    Owing to the choice of a large $T$, any new Reeb orbit that appears in the perturbations has to be of a different free homotopy class than the lowest action lifts of critical points considered in Theorem \ref{thm:comp}. Thus, new holomorphic cylinders  in the non-degenerate choices of data have topological obstructions from having an output at low-action Reeb orbits. This ends the proof since the above argument shows that the homology of the chain complexes $CC\left(\Sigma(2,4,5),\lambda^{H/1-H}_\circ, J^{H/1-H}\right)$ supported over $\cS_{tor}$ will be different.

    % $$CCH^{<T}(\Sigma(2,4,5),(1+\epsilon_T\pi^*H)\lambda) = CCH^{<T} (\Sigma(2,4,5), \lambda_\circ^{H})$$
    %  $$CCH^{<T}(\Sigma(2,4,5),(1+\epsilon_T\pi^*(1-H))\lambda) = CCH^{<T} (\Sigma(2,4,5), \lambda_\circ^{1-H})$$
    
\end{proof}

\begin{remark}\label{refl}[Reflections on invariance] The virtual techniques in \cite{pardon2019contact,chanda2025contact} can be used to show that the moduli space $\Mbar^+(\gamma_+,\gamma_-)$ are derived manifolds.  This follows from the observation that there are no non-trivial finite order automorphisms of a domain that fix a point and an asymptotic marker at the point. From now on we will drop the reference to Reeb orbits for notational clarity.

Although there is no isotropy in the moduli spaces, what goes awry about invariance under integer coefficients is that SFT compactness theorem does not give us a dMan equivalence of the form $$\partial \Mbar^+ \not\cong \Mbar^+ \times \Mbar^+ .$$Rather we get a correspondence  
\[
\begin{tikzcd}
& \Mbar^\pm \times \Mbar^+
\arrow[dl, "\glue"]
\arrow[dr, "\forget"]
& \\
\partial\Mbar^+
& &
\Mbar^+ \times \Mbar^+
\end{tikzcd}
\]
where the maps glue and forget are quotients by \textit{different} free actions by a cyclic group. Thus, to obtain true integral invariants, one needs to obtain (relative) virtual fundamental cycles that behave nicely with group actions. For example, if one wants to use abstract sections to obtain transversality, although one might achieve abstract transversality by sections instead of multisections for each of the lower spaces in the above diagram, the sections might not match under the pullback. In certain cases with automatic transversality one might get lucky and have integral differentials, but the whole invariance package does not go through. 

Another perspective is that cylindrical contact homology can be viewed as semi-infinite Morse homology on the orbifold $LY^\circ/S^1$ where $LY^\circ$ is the free loop space of $Y$ without the contractible loops and the $S^1$ action is by reparameterization. This orbifold has only finite isotropy occurring from multiply covered loops. The classical contact homology can be viewed as the orbifold Morse homology introduced in \cite{ChoHong2014} and it is known that their construction is not an invariant over $\bZ$. To obtain a Morse homology for an orbifold that is invariant over $\bZ$ coefficients, one way is to perform a Borel construction, see \cite{corrigan2020morse}. The corresponding construction in cylindrical contact homology for such a model would be the $S^1$ equivariant version introduced in \cite{hutchings2022s}.  It is interesting to note the discussion at the end of the blog-post \cite{Hutchings2014IntegerCoefficients} has a similar perspective, although I was unaware of this at the time of writing.

%, 

% Thus although one does not need multisections to define the counts in the differential, to show the differ

\end{remark}